\documentclass[12pt,a4paper]{amsart}
\usepackage{amssymb,amscd}

\theoremstyle{plain}
\newtheorem{theorem}{Theorem}[section]
\newtheorem{proposition}[theorem]{Proposition}
\newtheorem{lemma}[theorem]{Lemma}

\theoremstyle{remark}
\newtheorem*{remark}{Remark}

\def\free{\ensuremath{\mathcal L}}

\begin{document}

\title{The second homology group of current Lie algebras}
\author[P.~Zusmanovich]{Paul Zusmanovich}
\thanks{\noindent S.~M.~F. Ast\'erisque 226 (1994)}

\maketitle

\setcounter{section}{-1}

\section{Introduction}
It is a well known fact that the current Lie algebra 
$\mathcal G \otimes \mathbb C[[t, t^{-1}]]$ associated
to a simple finite-dimensional Lie $\mathbb C$-algebra $\mathcal G$ 
has a central extension leading to the affine non-twisted Kac-Moody algebra 
$\mathcal G \otimes \mathbb C[[t, t^{-1}]] \oplus  \mathbb Cz$ with bracket

\begin{equation}\notag
\{x \otimes f, y \otimes g\} = [x, y] \otimes fg + (x, y)Res \frac{df}{dt}g \> z
\end{equation}

\noindent where $( \cdot \,, \cdot )$ is the Killing form on $\mathcal G$ (cf. \cite{Kac}).

In view of the known relationship between central extensions and the second
(co)homology group with coefficients in the trivial module, one of the
main results of this paper can be considered as a generalization of this fact
for general current Lie algebras, i.e., Lie algebras of the form $L \otimes A$, where
$L$ is a Lie algebra and $A$ is associative commutative algebra, equipped with
bracket
\begin{equation}\notag
[x \otimes a, y \otimes b] = [x, y] \otimes ab.
\end{equation}

\begin{theorem}\label{0.1}
Let $L$ be an arbitrary Lie algebra over a field $K$ of characteristic $p \ne 2$ 
and $A$ an associative commutative algebra with unit over $K$.
Then there is an isomorphism of $K$-vector spaces:
\begin{multline}\tag{0.1}
H_2(L \otimes A) \simeq H_2(L) \otimes A \>\oplus\> B(L) \otimes HC_1(A)   \\
\oplus\> \wedge^2 (L/[L,L]) \otimes Ker(S^2(A) \to A) \>\oplus\> S^2(L/[L,L]) \otimes T(A)
\end{multline}
where the mapping $S^2(A) \to A$ induced by multiplication in $A$ and 
$T(A) = \langle ab \wedge c + ca \wedge b + bc \wedge a \,|\, a, b, c \in A \rangle$.
\end{theorem}

Here $B(L)$ is the space of coinvariants of the $L$-action on $S^2(L)$, $HC_1(A)$
is the first-order cyclic homology group of $A$, and $\wedge^2$ and $S^2$ denote the skew and
symmetric products, respectively. Notice that in the case $L = [L, L]$, the third
and fourth terms in the right-hand side of (0.1) vanish.


Many particular cases of this theorem were proved by different authors
previously. An exhaustive description of all previous works on this theme
may be found in \cite{H} and \cite{S}.

For the first time, a cohomology formula of the type (0.1) has appeared in
\cite{S}, where Theorem 0.1 was proved assuming that $L$ is 1-generated over an
augmentation ideal of its enveloping algebra. A. Haddi \cite{H} obtained a result
similar to Theorem 0.1 in the case where $K$ is a field of characteristic zero
(however, it seems that his arguments work over any field of characteristic
$p \ne 2,3$).

Our method of proof differs from all previous ones and is based on the Hopf
formula expressing $H_2(L)$ in terms of a presentation 
$0 \to I \to \free(X) \to L \to 0$,
where $\free = \free(X)$ is the free Lie algebra over $K$
freely generated by the set $X$:

\bigskip
\begin{equation}\tag{0.2}
H_2(L) \simeq ([\free, \free] \cap I)/[\free, I]
\end{equation}
\bigskip

\noindent (see, for example, \cite{KS}).

The contents of the paper are as follows. \S 1 is devoted to some technical
preliminary results. In \S 2 we determine the presentation of a current Lie
algebra $L \otimes A$. In \S 3 Theorem 0.1 is proved. As it corollary we get in \S 4
a description of the space $B(L \otimes A)$. In \S 5 a "noncommutative version" of
Theorem 0.1 is proved (Theorem 5.1). Namely, we derive the formula for the
second homology group of the Lie algebra $(A \otimes B)^{(-)}$, where $A$, $B$ 
are associative
(noncommutative) algebras with unit, and (--) in superscript denotes passing
to the associated Lie algebra. The technique used here is no longer based on the
Hopf formula, but on more or less direct computations in some factorspaces
of cycles. However, arguments used in proof, resemble, to a great extent, the
previous ones. Getting a particular case $B = M_n(K)$, we recover, after a
slight modification, an isomorphism $H_2(sl_n(A)) \simeq HC_1(A)$ obtained in \cite{KL}.

The following notational convention will be used: the letters $a, b, c, \dots$,
possibly with sub- and superscripts, denote elements of algebra $A$, while letters 
$u, v, w, \dots$ denote elements of the free Lie algebra $\free(X)$ with the set of
generators $X = \{ x_i \}$, if the otherwise is not stated. $\free^n(X)$ denotes 
the $n$th term in the derived series of $\free(X)$. 
The arrows $\rightarrowtail$ and $\twoheadrightarrow$ denote injection
and surjection, respectively.

All other undefined notions and notation are standard, and may be found,
for example, in \cite{F} for Lie algebra (co)homology, and in \cite{LQ} for cyclic 
homology. In some places we use diagram chasing and $3 \times 3$-Lemma without explicitly
mentioning it.

\section*{Acknowledgements}

This work was started during my graduate studies in the Institute of Mathematics 
and Mechanics of the Kazakh Academy of Sciences and finished in
the Bar-Ilan University. My thanks are due to Askar Dzhumadil'daev and Steven
Shnider, my former and present supervisors, for various sort of help, mathematical 
and other, and for interest in my work, and specially to Steven
Shnider for careful reading of preliminary versions of manuscript. Also the
financial support of Bar-Ilan University and Bat-Sheva Rotshild foundation
is gratefully acknowledged. Finally, I am grateful to the organizers of the
Conference on K-Theory held in June-July 1992 in Strasbourg for invitation
and opportunity to present there the results of this paper.

\section{Preliminaries}

Looking at formula (0.1), one can distinguish between the first two
``principal'' terms and other two ``non-principal'' ones. In order to simplify
calculations, we will obtain a variant of the Hopf formula leading to the appearance 
of ``principal'' terms only, and then the general case will be derived.

Each nonperfect Lie algebra $L$, i.e., not coinciding with its commutant
$[L, L]$, possesses a ``trivial'' homology classes of $2$-cycles with coefficients in the
module $K$, namely, classes whose representatives do not lie in $L \wedge [L,L]$. More
precisely, consider a natural homomorphism 
$\psi: H_2(L) \to H_2(L/[L, L]) \simeq \wedge^2 (L/[L, L])$ and denote 
$H_2^{ess}(L) = Ker\psi$, the homology classes of ``essential'' cycles.

\begin{lemma}\label{1.1}
One has an exact sequence

\begin{equation}\notag
0 \to H_2^{ess}(L) \to H_2(L) \overset{\psi}\to \wedge^2(L/[L,L]) 
                              \overset{\pi}\to  [L,L]/[[L,L],L]   \to 0
\end{equation}

\noindent where $\pi$ is induced by multiplication in $L$.
\end{lemma}

\begin{proof}
This is just an obvious consequence of a 5-term exact sequence derived
from the Hochschild-Serre spectral sequence 
$H_n(L/[L,L], H_m([L,L])) \Rightarrow H_{n+m}(L)$.
\end{proof}

Further, we need a version of Hopf formula for $H_2^{ess}(L)$.

\begin{lemma}\label{1.2}
Given a presentation $0 \to I \to \free \to L \to 0$ of a Lie algebra $L$, one has

\begin{equation}\tag{1.1}
H_2^{ess}(L) \simeq \frac{\free^3 \cap I}{\free^3 \cap [\free, I]}.
\end{equation}
\end{lemma}


\begin{proof}
Since $L/[L,L] \simeq \free/(\free^2 + I)$, the Hopf formula (0.2) being applied to
the algebra $L/[L,L]$ gives 
$H_2(L/[L,L]) \simeq \free^2/[\free, \free^2 + I]$, and

\begin{equation}\notag
Ker\psi = 
Ker \left( \frac{\free^2 \cap I}{[\free, I]} \to 
    \frac{\free^2}{[\free, \free^2 + I]} \right) \simeq
\frac{\free^2 \cap I \cap [\free, \free^2 + I]}{[\free, I]} \simeq
\frac{\free^3 \cap I}{\free^3 \cap [\free, I]}. 
\end{equation}
\end{proof}

Now consider an action of a Lie algebra $L$ on $S^2(L)$ via

\begin{equation}\notag
[z, x \vee y] = [z,x] \vee y + x \vee [z,y].
\end{equation}

\noindent Let $B(L) = S^2(L)/[L, S^2(L)]$ be the space of coinvariants of this action. The
dual $B(L)^*$ is the space of symmetric bilinear invariant forms on $L$.

Let $I, J$ be ideals of $L$. Define $B(I, J)$ to be the space of coinvariants of
the action of $L$ on $I \vee J$. One has a natural embedding $B(I, J) \to B(L)$. The
natural map $L \vee J \to (L/I) \vee ((I + J)/I)$ defines a surjection 
$B(L, J) \to B(L/I, (I + J)/ I)$.

\begin{lemma}\label{1.3}
The short sequence

\begin{equation}\tag{1.2}
0 \to B(L, I \cap J) + B(I,J) \to B(L,J) \to B(L/I,(I+J)/I) \to 0
\end{equation}

\noindent is exact.
\end{lemma}

\begin{proof}
Since $Ker (L \vee J \to L/I \vee (I+J)/I) = L \vee (I \cap J) + I \vee J$, the
factorization through $[L, S^2(L)]$ yields

\begin{multline}\notag
Ker(B(L,J) \to B(L/I, (I+J)/I))   \\
= (L \vee (I \cap J) + I \vee J + [L, S^2(L)])/[L, S^2(L)] \simeq 
B(L, I \cap J) + B(I,J).
\end{multline}
\end{proof}

\noindent\textit{Remark}. 
Actually we need the following two cases of this Lemma:

(1) $J = [L,L]$. Since $I \vee [L,L]$ and $[I,L] \vee L$ are congruent modulo $[L, S^2(L)]$
and $[I,L] \subseteq I \cap [L,L]$, then $B(I,[L,L]) \subseteq B(L, I \cap [L, L])$ 
and we get a short exact sequence

\begin{equation}\tag{1.3}
0 \to B(L, I \cap [L,L]) \to B(L,[L,L]) \to B(L/I, [L/I, L/I]) \to 0.
\end{equation}

(2) $I = [L,L]$ and $J = L$. Then taking into account that for an abelian Lie
algebra $M$, $B(M) \simeq S^2(M)$, the short exact sequence (1.2) becomes

\begin{equation}\tag{1.4}
0 \to B(L,[L,L]) \to B(L) \to S^2(L/[L,L]) \to 0.
\end{equation}

\section{Presentation of $L \otimes A$}

In this section starting from a presentation of $L$ we construct a presentation 
of $L \otimes A$.

Let $0 \to I \to \free(X) \overset{p}\to L \to 0$ be a presentation of the Lie algebra $L$.
Tensoring by $A$, we get a short exact sequence

\begin{equation}\tag{2.1}
0 \to I \otimes A \to \free(X) \otimes A \overset{p\otimes 1}\to L \otimes A \to 0.
\end{equation}

Let $X(A)$ be a set of symbols $x(a), x \in X, a \in A$. Define a homomorphism
$\phi: \free(X(A)) \to \free(X) \otimes A$ by

\begin{equation}\notag
\phi: u(x_1(a_1), \dots, x_n(a_n)) \mapsto u(x_1, \dots, x_n) \otimes a_1 \dots a_n.
\end{equation}

\noindent 
Obviously this mapping is surjective, and taking into account (2.1), gives rise
to the following exact sequence:

\begin{equation}\tag{2.2}
0 \to \phi^{-1}(I \otimes A) \to \free(X(A))
  \overset{(p\otimes 1) \circ \phi}\longrightarrow L \otimes A \to 0
\end{equation}

\noindent which gives the presentation of $L \otimes A$.

In order to determine the structure of $\phi^{-1}(I \otimes A)$, let us introduce one 
notation. For each homogeneous element $u = u(x_1, \dots, x_n)$ of $\free(X)$, define 
$u(a)$ to be $u(x_1(a), x_2(1), \dots, x_n(1))$. Now having an arbitrary element 
$u \in \free(X)$, define $u(a)$ as $u_1(a) + \dots + u_k(a)$, where 
$u = u_1 + \dots + u_k$ is decomposition of $u$ into the sum of homogeneous components.

\begin{lemma}\label{2.1}
\hfill
\begin{enumerate}
\item
$Ker\phi$ is linearly generated by elements of the form
\begin{equation}\tag{2.3}
\sum_j u(x_{i_1}(a_1^j), \dots, x_{i_n}(a_n^j))
\end{equation}

\noindent where $u(x_{i_1}, \dots, x_{i_n})$ is homogeneous element of $\free(X)$ and
\break
$\sum_j a_1^j \dots a_n^j = 0$.

\item
$\phi^{-1}(I \otimes A)$ is linearly generated modulo $Ker\phi$ by elements of the form
$u(a)$, where $u \in I$.
\end{enumerate}
\end{lemma}

\begin{proof}
(1) Evidently each element of $\free(X(A))$ may be expressed as a sum
of elements of the form $u(a)$ and elements of the form (2.3), the latter lying in 
$Ker\phi$. To prove that they exhaust all $Ker\phi$, take a nonzero element
$\sum_i \sum_j u_i(a_{ij})$ belonging to $Ker\phi$, 
where $u_i$'s are linearly independent, and
obtain $\sum_i \sum_j u_i \otimes a_{ij} = 0$, which implies $\sum_j a_{ij} = 0$ 
for each $i$.

(2) The factorspace $\phi^{-1} (I \otimes A)/Ker\phi$, consisting from cosets 
$u(a) + Ker\phi$, maps onto $I \otimes A$, whence the conclusion.
\end{proof}

We also need the following technical result.


\begin{lemma}\label{2.2}
For any $u, v, w \in \free(X)$ and $a, b, c \in A$, the elements

\begin{equation}\notag
[[w, u](a), v(b)] - [[w, u](b), v(a)] + [[w, v](a), u(b)) - [[w, v](b), u(a)]
\end{equation}

\noindent and

\begin{align}
  & [[u, v](ab), w(c)] - [[u, v)(c), w(ab)]   \notag \\
+ & [[u, v](ca), w(b)] - [[u, v](b), w(ca)]   \notag \\
+ & [[u, v](bc), w(a)] - [[u, v](a), w(bc)]   \notag
\end{align}

\noindent belong to $[\free(X(A)), Ker\phi]$.
\end{lemma}

\begin{proof}
Consider the first case only, the second one is analogous. We have modulo
$[\free(X(A)), Ker\phi]$:

\begin{multline}\notag
[[w, u](a), v(b)] - [[w, u](b), v(a)] + [[w, v)(a), u(b)] - [[w, v](b), u(a)]  \\
\equiv  [[w(1), u(a)], v(b)] + [[v(b), w(1)], u(a)]                            \\
+       [[w(1), v(a)], u(b)] + [[u(b), w(1)], v(a)]                            \\
\equiv -[[u(a), v(b)], w(1)] + [[u(b), v(a)], w(1)] \equiv 0
\end{multline}
\end{proof}

\section{The second homology of $L \otimes A$}

The aim of this section is to prove Theorem 0.1.

Consider the following commutative diagram with exact rows and columns,
where $\phi^{-1}$ stands for $\phi^{-1}(I \otimes A)$
(we will use this notation in some places further):

{\footnotesize

$$
\minCDarrowwidth 9pt
\begin{CD}
   @. 0    @.       0             @.                                            \\
@.      @VVV        @VVV       @.                                               \\
0     @>>> \free^3(X(A)) \cap [\free(X(A)), \phi^{-1}] \cap Ker\phi @>>> 
\free^3(X(A)) \cap Ker\phi     @.                                               \\
   @.   @VVV        @VVV          @.                                            \\
0     @>>> \free^3(X(A)) \cap [\free(X(A)), \phi^{-1}] @>>> 
\free^3(X(A)) \cap \phi^{-1} @>>> H_2^{ess}(L \otimes A)  @>>> 0                \\
   @.   @V{\phi}VV  @V{\phi}VV    @.                                            \\
0     @>>>   (\free^3(X) \otimes A) \cap [\free(X) \otimes A, I \otimes A] @>>>
(\free^3(X) \otimes A) \cap (I \otimes A)    @.                                 \\
   @.   @VVV        @VVV          @.                                            \\
@.    0    @.       0          @.
\end{CD}$$

}


The middle row follows from the Lemma 1.2 applied to the presentation (2.2).

Completing this diagram to the third column, we get a short exact sequence

\begin{multline}\tag{3.1}
0 \to 
\frac{\free^3(X(A)) \cap Ker\phi}
{\free^3(X(A)) \cap [\free(X(A)), \phi^{-1}(I \otimes A)] \cap Ker\phi} 
\to H_2^{ess}(L \otimes A)           \\
\to \frac{\free^3(X) \cap I}{\free^3(X) \cap [\free(X), I]} \otimes A \to 0.
\end{multline}

According to Lemma 1.2, the right term here is nothing but
$H_2^{ess}(L) \otimes A$. Let us compute the left term.

Let $\mathcal F(Y)$ be a free skewcommutative algebra on an alphabet $Y$ with 
nonassociative product denoted by $[ \cdot \,, \cdot]$. Define a mapping 
$\alpha: \mathcal F^2(X(A)) \to S^2(\mathcal F(X)) \otimes (A \wedge A)$ by

\begin{multline}\tag{3.2}
\alpha: [u(x_1(a_1), \dots, x_n(a_n)), v(x_1(b_1), \dots, x_m(b_m))] \mapsto  \\
        (u(x_1, \dots, x_n) \vee v(x_1, \dots, x_m)) \otimes (a_1 \dots a_n \wedge b_1 \dots b_m).
\end{multline}

\noindent (recall that $\mathcal F^2(Y)$ is just $[\mathcal F(Y), \mathcal F(Y)]$).

It is easy to see that this mapping is well defined and surjective.

Let $J(Y)$ be an ideal of $\mathcal F(Y)$ generated by elements of the form 
$[[u, v], w] + [[w, u], v] + [[v, w], u], u, v, w \in \mathcal F(Y)$ such that
$\mathcal F(Y)/J(Y) \simeq \free(Y)$.

\begin{lemma}\label{3.1}
\begin{multline}\notag
\alpha(J(X(A))) = 
(J(X) \vee \mathcal F(X) + [\mathcal F(X), S^2(\mathcal F(X))]) \otimes (A \wedge A)  \\
+ (\mathcal F^2(X) \vee \mathcal F(X)) \otimes T(A).
\end{multline}
\end{lemma}

\begin{proof}
Writing the generic element in $J(X(A))$, it is easy to see, by considering graded degree,
that every element in $\alpha(J(X(A)))$ can be written as a sum
of an element lying in $(J(X) \vee \mathcal F(X)) \otimes (A \wedge A)$ and an element of the form

\begin{equation}\tag{3.3}
([u, v] \vee w) \otimes (ab \wedge c) + 
([w, u] \vee v) \otimes (ca \wedge b) + 
([v, w] \vee u) \otimes (bc \wedge a)
\end{equation}

\noindent for certain $u, v, w \in \mathcal F(X)$ and $a, b, c \in A$.

Substituting in (3.3) $b = c = 1$, we get an element

\begin{equation}\notag
([u,v] \vee w + [w,u] \vee v - [v,w] \vee u) \otimes (1 \wedge a).
\end{equation}

\noindent Now permuting the letters $u, v$ in the last expression, one easily get

\begin{equation}\notag
(\mathcal F^2(X) \vee \mathcal F(X)) \otimes (1 \wedge A) \subset \alpha(J(A(X))).
\end{equation}


\noindent Substituting in (3.3) $c = 1$ and taking into account the last relation, we get

\begin{equation}\tag{3.4}
([w, u] \vee v + u \vee [w, v]) \otimes (A \wedge A) \subset \alpha(J(A(X))).
\end{equation}

\noindent Any element in (3.3) is congruent modulo (3.4) to an element of the form
\begin{equation}\notag
(\mathcal F^2(X) \vee \mathcal F(X)) \otimes (ab \wedge c + ca \wedge b + bc \wedge a)
\end{equation}

\noindent proving the Lemma.
\end{proof}

Now factoring the surjection $\alpha$ through $J(X(A))$ and using Lemma 3.1, we
get a mapping
\begin{equation}\notag
\overline\alpha: \free^2(X(A)) \longrightarrow 
                 B(\free(X)) \otimes HC_1(A) + (KX \vee KX) \otimes (A \wedge A),
\end{equation}

\noindent ($KX$ denotes the space of linear terms in $\mathcal F(X)$ such that 
$\mathcal F(X) = KX + \mathcal F^2(X)$), which being restricted to $\free^3(X(A))$, 
gives rise to the surjection

\begin{equation}\notag
\overline\alpha: \free^3(X(A)) \longrightarrow B(\free(X), \free^2(X)) \otimes HC_1(A),
\end{equation}

\noindent where $HC_1(A) = (A \wedge A)/T(A)$ is a first order cyclic homology of $A$.

Further, the restriction of the mapping $\phi$ defined in \S 2 to $\free^3(X(A))$ leads to
a surjection $\phi: \free^3(X(A)) \to \free^3(X) \otimes A$.

\begin{lemma}\label{3.2}
$\overline\alpha(\free^3(X(A)) \cap Ker\phi) = \overline\alpha(\free^3(X(A)))$.
\end{lemma}

\begin{proof}
The Lemma follows immediately from Lemma 2.1 and equality

\begin{equation}\notag
\overline\alpha [u(a), v(b)] = \frac 12 \overline\alpha ([u(a), v(b)] - [u(b), v(a)]),
\end{equation}

\noindent where the argument in the right-hand side lies in $Ker\phi$.
\end{proof}

\begin{lemma}\label{3.3}
\begin{equation}\notag
\overline\alpha (\free^3(X(A)) \cap [\free(X(A)), \phi^{-1}(I \otimes A)]) = 
B(\free(X), I \cap \free^2(X)) \otimes HC_1(A) .
\end{equation}
\end{lemma}

\begin{proof}
According to Lemma 2.1, $\overline\alpha(\free^3(X(A)) \cap [\free(X(A)), \phi^{-1} (I \otimes A)])$
consists from the linear span of the following elements:

\begin{equation}\notag
\overline{u \vee v} \otimes \overline{a \wedge b}
\end{equation}

\noindent where either $u \in \free^2(X), v \in I$ or $u \in \free(X), v \in I \cap \free^2(X)$, and

\begin{equation}\notag
\sum_j \overline{u \vee v} \otimes \overline{a \wedge b_j}
\end{equation}


\noindent where $\sum_j b_j = 0$. The last expression obviously vanishes.

Modulo $[\free(X), S^2(\free(X))]$ we have:

\begin{equation}\notag
\free^2(X) \vee I \equiv \free(X) \vee [I, \free(X)] \subseteq \free(X) \vee (I \cap \free^2(X)),
\end{equation}

\noindent which implies the assertion of Lemma.
\end{proof}

Lemma 3.3 implies that the mapping $\overline\alpha$, being restricted to \break
$\free^3(X(A)) \cap \phi^{-1}(I \otimes A)$ and factored through 
$\free^3(X(A)) \cap [\free(X(A)), \phi^{-1} (I \otimes A)]$,
gives rise to a surjection

\begin{equation}\tag{3.5}
\beta: \frac{\free^3(X(A)) \cap \phi^{-1}(I \otimes A)}
            {\free^3(X(A)) \cap [\free(X(A)), \phi^{-1}(I \otimes A)]} \to
       \frac{B(\free(X), \free^2(X))}{B(\free(X), I \cap \free^2(X))} \otimes HC_1(A).
\end{equation}

The right-hand side here is by (1.3) isomorphic to $B(L,[L,L]) \otimes HC_1(A)$.

Further, according to Lemma 3.2, $\beta$ can be restricted to a surjection
\begin{equation}\tag{3.6}
\beta: \frac{\free^3(X(A)) \cap Ker\phi}
            {\free^3(X(A)) \cap [\free(X(A)), \phi^{-1}(I \otimes A)] \cap Ker\phi} \to
       B(L,[L,L]) \otimes HC_1(A).
\end{equation}

\begin{lemma}\label{3.4}
$\beta$ in (3.6) is injective.
\end{lemma}

\begin{proof}
Denoting the left-hand side in (3.5) as $Frac$, consider the following diagram:

{\small
$$
\minCDarrowwidth 15pt
\begin{CD}
Ker([L,L] \otimes A \wedge L \otimes A \to L \otimes A) @>h>> H_2^{ess}(L \otimes A) @>j>> Frac  \\
@AiAA                                                         @.                           @V{\beta}VV \\
L \vee [L,L] \otimes A \wedge A @>n>>       {}  @. B(L,[L,L]) \otimes HC_1(A)
\end{CD}$$
}

\noindent where $h$ is the obvious factorization, $j$ is the isomorphism following from Lemma 1.2
applied to presentation (2.2), $n = l \otimes s$, where $l: L \vee [L, L] \to B(L,[L, L])$
and $s: A \wedge A \to HC_1(A)$ are obvious factorizations, and $i$ is defined as

\begin{equation}\tag{3.7}
i: (x \vee y) \otimes (a \wedge b) \mapsto 
   \frac 12 (x \otimes a \wedge y \otimes b - 
             x \otimes b \wedge y \otimes a)
\end{equation}

\noindent for $x \in [L, L], y \in L$.


The following calculation verifies the commutativity of this diagram:

\begin{align}
&  \beta \circ j \circ h \circ i ((x \vee y) \otimes (a \wedge b))        \notag \\
& = \frac 12 
  \beta \circ j \circ h   (x \otimes a \wedge y \otimes b - 
                           x \otimes b \wedge y \otimes a)                \notag \\
& = \frac 12 
  \beta \circ j (\overline{x \otimes a \wedge y \otimes b - 
                           x \otimes b \wedge y \otimes a})               \notag \\
& = \frac 12
  \beta \circ j (\overline{(u(a) + \phi^{-1}) \wedge (v(b) + \phi^{-1}) -
                           (u(b) + \phi^{-1}) \wedge (v(a) + \phi^{-1})}) \notag \\
& = \frac 12 \beta (\overline{[u(a), v(b)] - [u(b), v(a)]})               \notag \\
& = \frac 12 (\overline{(x \vee y) \otimes (a \wedge b) - 
                        (x \vee y) \otimes (b \wedge a)})                 \notag \\
& = \overline{x \vee y} \otimes \overline{a \wedge b}                     \notag \\
& = n ((x \vee y) \otimes (a \wedge b))                                   \notag 
\end{align}

\noindent where the overlined elements denote cosets in the corresponding factorspaces,
and $x = u + I, y = v + I$.

It is also clear from the previous calculation and Lemmas 2.1 and 3.2 that
the image of $j \circ h \circ i$ coincides with the left-hand side of (3.6).

Thus the kernel of the mapping (3.6) can be evaluated as

\begin{align}
Ker\beta &= j \circ h \circ i (Ker\, n)                                         \notag \\
        &= j \circ h \circ i (\langle [z,x] \vee y + [z,y] \vee x \rangle 
                              \otimes \langle a \wedge b \rangle               \notag \\
        &+ \langle [x,y] \vee z \rangle \otimes 
                 \langle ab \wedge c + ca \wedge b + bc \wedge a \rangle)      \notag \\
        &= j (\langle \overline{[z,x] \otimes a  \wedge y \otimes b -
                                [z,x] \otimes b  \wedge y \otimes a}           \notag \\
        &+            \overline{[z,y] \otimes a  \wedge x \otimes b - 
                                [z,y] \otimes b  \wedge x \otimes a} \rangle   \notag \\
        &+  \langle   \overline{[x,y] \otimes ab \wedge z \otimes c -
                                [x,y] \otimes c  \wedge z \otimes ab}          \notag \\
        &+            \overline{[x,y] \otimes ca \wedge z \otimes b - 
                                [x,y] \otimes b  \wedge z \otimes ca}          \notag \\
        &+            \overline{[x,y] \otimes bc \wedge z \otimes a -
                                [x,y] \otimes a  \wedge z \otimes bc} \rangle) \notag \\
        &= \langle \overline{[[w,u](a),  v(b)] - [[w,u](b), v(a)]}             \notag \\
        &+         \overline{[[w,v](a),  u(b)] - [[w,v](b), u(a)]} \rangle     \notag \\
        &+ \langle \overline{[[u,v](ab), w(c)] - [[u,v](c), w(ab)]}            \notag \\
        &+         \overline{[[u,v](ca), w(b)] - [[u,v](b), w(ca)]}            \notag \\
        &+         \overline{[[u,v](bc), w(a)] - [[u,v](a), w(bc)]} \rangle    \notag
\end{align}


\noindent (here $u = x + I, v = y + I, w = z + I$). The latter expression vanishes thanks
to Lemma 2.2.
\end{proof}

Putting together (3.1), (3.6) and Lemma 3.4, we get

\begin{proposition}\label{3.5}
$H_2^{ess}(L \otimes A) \simeq H_2^{ess}(L) \otimes A \>\oplus\> B(L,[L,L]) \otimes HC_1(A)$.
\end{proposition}

By Lemma 1.1 we have an exact sequence

\begin{multline}\tag{3.8}
0 \to H_2^{ess}(L \otimes A) \to H_2(L \otimes A) \to \wedge^2(L/[L, L] \otimes A) 
  \overset{\pi_A}\to [L,L]/[[L,L],L] \otimes A  \\
  \to 0.
\end{multline}

\begin{lemma}\label{3.6}
\begin{multline}\notag
Ker\pi_A \simeq Ker(\wedge^2(L/[L,L]) \overset{\pi}\to [L,L]/[[L,L],L]) \otimes A   \\
         \oplus\> \wedge^2(L/[L, L]) \otimes Ker(S^2(A) \to A) 
       \>\oplus\> S^2(L/[L,L]) \otimes \wedge^2(A).
\end{multline}
\end{lemma}

\begin{proof}
The following commutative diagram with exact rows and columns

{\footnotesize
$$
\minCDarrowwidth 9pt
\begin{CD}
  @.                              @.          0                            \\    
@.     @.                                     @VVV            \\
  @.   S^2(L/[L,L]) \otimes \wedge^2(A) @>=>> S^2(L/[L,L]) \otimes \wedge^2(A)  @>{\pi}>>   0  \\
@.     @.                                     @ViVV                                    @VVV  \\
0 @>>> Ker\pi_A                         @>>>  \wedge^2(L/[L,L] \otimes A)       @>{\pi_A}>> [L,L]/[[L,L],L] \otimes A @>>> 0    \\
@.     @.                                     @VkVV                                        @|  \\ 
0 @>>> Ker(\pi \otimes m)               @>>>  \wedge^2 (L/[L,L]) \otimes S^2(A) @>{\pi \otimes m}>> [L,L]/[[L,L],L] \otimes A @>>> 0    \\
@.     @.                                     @VVV   \\
  @.                              @.          0      
\end{CD}$$
}

\noindent where $i$ is defined in (3.7), and

\begin{align}
k: x \otimes a \wedge y \otimes b &\mapsto (x \wedge y) \otimes (a \vee b)  \notag \\
m: a \vee b &\mapsto ab  \notag
\end{align}

\noindent for $x, y \in L/[L,L], a, b \in A$, implies

\begin{equation}\tag{3.9}
Ker\pi_A \simeq Ker(\pi \otimes m) \>\oplus\> S^2(L/[L,L]) \otimes \wedge^2(A).
\end{equation}


Considering the commutative diagram with exact rows and columns

{\footnotesize
$$
\minCDarrowwidth 9pt
\begin{CD}
  @.                              @.          0                            \\    
@.     @.                                     @VVV            \\
  @.   \wedge^2(L/[L,L]) \otimes Ker\,m @>=>> \wedge^2(L/[L,L]) \otimes Ker\,m  @>{\pi}>>   0  \\
@.     @.                                     @VVV                                    @VVV  \\
0 @>>> Ker(\pi \otimes m)                         @>>>  \wedge^2(L/[L,L]) \otimes S^2(A)       @>{\pi \otimes m}>> [L,L]/[[L,L],L] \otimes A @>>> 0    \\
@.     @.                                     @V{1 \otimes m}VV                                        @|  \\ 
0 @>>> Ker(\pi \otimes 1)               @>>>  \wedge^2(L/[L,L]) \otimes A @>{\pi \otimes 1}>> [L,L]/[[L,L],L] \otimes A @>>> 0    \\
@.     @.                                     @VVV   \\
  @.                              @.          0      
\end{CD}$$
}

\noindent we get

\begin{multline}\tag{3.10}
Ker(\pi \otimes m) \simeq \wedge^2(L/[L,L]) \otimes Ker(S^2(A) \to A)     \\
                   \oplus\> Ker(\wedge^2(L/[L,L]) \overset{\pi}\to [L,L]/[[L,L],L]) \otimes A.
\end{multline}

Putting (3.9) and (3.10) together proves the Lemma.
\end{proof}

Combining Proposition 3.5, (3.8) and Lemma 3.6, we get

\begin{multline}\notag
H_2(L \otimes A) \simeq H_2^{ess}(L) \otimes A 
                 \>\oplus\> Ker(\wedge^2(L/[L,L]) \to [L,L]/[L,[L,L]]) \otimes A   \\
                   \oplus\> B(L,[L,L]) \otimes HC_1(A) 
                 \>\oplus\> S^2(L/[L,L]) \otimes \wedge^2(A)                       \\  
                   \oplus\> \wedge^2(L/[L,L]) \otimes Ker(S^2(A) \to A).
\end{multline}

\noindent By Lemma 1.1 the first two terms here give $H_2(L) \otimes A$. Using a (noncanonical)
splitting $\wedge^2(A) = HC_1(A) \oplus T(A)$ and the exact sequence (1.4), the third and
fourth terms give $B(L) \otimes HC_1(A) \oplus S^2(L/[L,L]) \otimes T(A)$. Combining these
identifications gives Theorem 0.1.

\remark It is interesting to compare Theorem 0.1 with the two-dimensional
case of the homological operation

\begin{equation}\notag
H_n(L \otimes A) \to \bigoplus_{i+j=n-1} HC_i(U(L)) \otimes HC_j(A)
\end{equation}

\noindent defined in \cite{FT} ($U(L)$ is the universal enveloping algebra of $L$ and the ground
field assumed to be of characteristic zero). Taking $n = 2$, we obtain a mapping

\begin{equation}\tag{3.11}
H_2(L \otimes A) \to HC_1(U(L)) \otimes HC_0(A) \>\oplus\> HC_0(U(L)) \otimes HC_1(A).
\end{equation}


Cyclic homology of universal enveloping algebras was studied in \cite{FT} and
\cite{Kas2}. Using their results, we may observe that if $S(L)$ denotes the whole
symmetric algebra over $L$, then

\begin{equation}\notag
HC_0(U(L)) = H_0(L, S(L)) = S(L)/[L,S(L)]
\end{equation}

\noindent and $HC_1(U(L))$ is a certain factorspace of $H_1(L, S(L))$ containing $H_2(L)$.
This implies that in general (3.11) is neither injection, nor surjection. However, 
if $L = [L,L]$, then (3.11) is an injection.

\section{Computation of $B(L \otimes A)$}

Theorem 0.1 allows us to compute $B(L \otimes A)$ in terms of $L$ and $A$ (of
course, an alternative but longer proof may be given by means of direct computations).

\begin{theorem}\label{4.1}
$B(L \otimes A) \simeq B(L,[L,L]) \otimes A \>\oplus\> S^2(L/[L, L] \otimes A)$.
\end{theorem}

\begin{proof}
It is more convenient to use Proposition 3.5 rather then Theorem 0.1
to obtain a formula for $B(L \otimes A, [L,L] \otimes A)$ and then to derive the general
case.

Take any commutative unital algebra $A^{\prime}$ with $HC_1(A^{\prime}) \simeq K$. According
to Proposition 3.5,

\begin{multline}\tag{4.1}
H_2^{ess}(L \otimes A \otimes A') \simeq H_2^{ess}(L \otimes A) \otimes A^{\prime} 
                                  \>\oplus\> B(L \otimes A, [L,L] \otimes A)       \\
                                  \simeq H_2^{ess}(L) \otimes A \otimes A^{\prime} 
                                  \>\oplus\> B(L,[L,L]) \otimes HC_1(A) \otimes A^{\prime}
                                  \>\oplus\> B(L \otimes A, [L,L] \otimes A).
\end{multline}

\noindent On the other hand,

\begin{multline}\tag{4.2}
H_2^{ess}(L \otimes A \otimes A') \simeq H_2^{ess}(L) \otimes A \otimes A^{\prime} 
                                  \>\oplus\> B(L,[L,L]) \otimes HC_1(A \otimes A^{\prime})  \\
                                  \simeq H_2^{ess}(L) \otimes A \otimes A^{\prime} 
                                  \>\oplus\> B(L,[L,L]) \otimes HC_1(A) \otimes A^{\prime}
                                  \>\oplus\> B(L,[L,L]) \otimes A.
\end{multline}

\noindent (the last isomorphism follows from the partial first-order commutative case
of the K\"unneth formula for cyclic homology (cf. \cite{Kas1}): 
$HC_1(A \otimes A') \simeq HC_1(A) \otimes A^{\prime} + A \otimes HC_1(A^{\prime})$).

Comparing (4.1) and (4.2), and using the naturality condition guaranteeing
compatibility, one has

\begin{equation}\notag
B(L \otimes A, [L,L] \otimes A) \simeq B(L,[L,L]) \otimes A.
\end{equation}

Now the assertion of Theorem easily follows from the last isomorphism and
the short exact sequence (1.4) applied to the algebra $L \otimes A$.
\end{proof}


\section{The second homology of $A \otimes B$}

Recall that given an associative algebra $A$, we may consider its associated
Lie algebra $A^{(-)}$ with the same underlying space $A$ and the bracket 
$[a,b] = ab-ba$, as well as a Jordan algebra $A^{(+)}$ with multiplication 
$a \circ b = \frac 12 (ab+ba)$.

Recall that 
$T(A) = \langle ab \wedge c + ca \wedge b + bc \wedge a | a, b, c \in A \rangle$. 
For the sake of convenience we will also use the following notation:

\begin{align}
   T (A,[A,A]) &= \frac{T(A) + [A,A] \wedge A}{[A,A] \wedge A}   \notag \\
HC_1 (A,[A,A]) &= \frac{A \wedge A}{[A, A] \wedge A + T(A)}      
            \simeq \frac{\wedge^2(A/[A,A])}{T(A,[A,A])}         \notag 
\end{align}

\noindent (the second one is an analogue of $H_2^{ess}(L)$ for cyclic homology).

The aim of this section is to prove the following

\begin{theorem}\label{5.1}
Let $A, B$ be associative algebras with unit over a field $K$ of
characteristic $p \ne 2$. Let $F(A,B)$ denote the direct sum of the following four
vector spaces:

\begin{enumerate}
\item $A[A,A]/[A,A] \otimes HC_1(B)$
\item $A/A[A,A] \otimes H_2(B^{(-)})$
\item $(Ker(S^2(A) \to A/[A,A]))/[A,S^2(A)] \otimes HC_1(B,[B,B])$
\item $Ker(S^2(A/[A,A]) \to A/A[A,A]) \otimes T(B,[B,B])$
\end{enumerate}

\noindent where arrows in (3) and (4) are induced by (associative or Jordan) multiplication
in $A$.

Then $H_2((A \otimes B)^{(-)}) \simeq F(A,B) \oplus F(B,A)$.
\end{theorem}

The proof is divided into several steps.

We employ the following short exact sequence:

\begin{equation}\notag
0 \to \wedge^2 A \otimes S^2B \overset{i}\to \wedge^2(A \otimes B) 
\overset{p}\to S^2A \otimes \wedge^2 B \to 0
\end{equation}

\noindent where the middle term is identified with the direct sum of two extreme ones
via

\begin{equation}\notag
a_1 \otimes b_1 \wedge a_2 \otimes b_2 \leftrightarrow 
a_1 \wedge a_2 \otimes b_1 \vee b_2 + a_1 \vee a_2 \otimes b_1 \wedge b_2,
\end{equation}

\noindent and $i$ and $p$ are obvious imbedding and projection respectively. In what
follows, this will be used without explicitly mentioning it.

The arguments are quite analogous to the ones at the beginning of \S 3. Here
they applied to $H_2((A \otimes B)^{(-)}) \simeq Ker\,d/Im\,d$ ($d$ is the differential in the
standard homology complex of $(A \otimes B)^{(-)}$). The mapping $p$ gives rise to the
following short exact sequence:

\begin{equation}\tag{5.1}
0 \to \frac{Ker\,p \cap Ker\,d}{Ker\,p \cap Im\,d} \to H_2((A \otimes B)^{(-)}) 
  \to \frac{p(Ker\,d)}{p(Im\,d)} \to 0.
\end{equation}


The Lie bracket on $A \otimes B$ may be written as a sum

\begin{equation}\notag
[a_1 \otimes b_1, a_2 \otimes b_2] = 
[a_1,a_2] \otimes b_1 \circ b_2 + a_1 \circ a_2 \otimes [b_1, b_2].
\end{equation}

The proof of the following statement is quite analogous to the proof of (3.10).

\addtocounter{theorem}{-1}
\begin{lemma}\label{lemma5.1}
\begin{equation}\notag
p(Ker\,d) \simeq 
Ker(A \vee A \to A/[A,A]) \otimes B \wedge B + A \vee A \otimes Ker(B \wedge B \to B)
\end{equation}

\noindent where the first arrow is induced by (associative or Jordan) multiplication in
algebra $A$ and the second one is the Lie multiplication in $B^{(-)}$.
\end{lemma}

\begin{lemma}\label{5.2}
$p(Im\, d)$ is a linear span of the following elements:

\begin{enumerate}
\item $[A, S^2(A)] \otimes B \wedge B$
\item $[A,A] \vee A \otimes T(B)$
\item $(a_1 \vee a_2 - 1 \vee a_1 \circ a_2) \otimes [B,B] \wedge B, a_i \in A$
\item $A \vee A \otimes ([b_1,b_2] \wedge b_3 + [b_3,b_1] \wedge b_2 + [b_2,b_3] \wedge b_1), 
      b_i \in B$.
\end{enumerate}
\end{lemma}

\begin{proof}
We adopt the notation $x \equiv 0$ denoting the fact that certain element $x$
of $A \vee A \otimes B \wedge B$ lies in $p(Im\,d)$. The generic relation defining the quotient
by $p(Im\,d)$ is

\begin{align}
  &[a_1,a_2]       \vee a_3 \otimes (b_1 \circ b_2) \wedge b_3 
+  (a_1 \circ a_2) \vee a_3 \otimes [b_1,b_2]       \wedge b_3            \notag \\
+ &[a_3,a_1]       \vee a_2 \otimes (b_3 \circ b_1) \wedge b_2 
+  (a_3 \circ a_1) \vee a_2 \otimes [b_3,b_1]       \wedge b_2            \tag{5.2} \\
+ &[a_2,a_3]       \vee a_1 \otimes (b_2 \circ b_3) \wedge b_1 
+  (a_2 \circ a_3) \vee a_1 \otimes [b_2,b_3]       \wedge b_1 \equiv 0.  \notag
\end{align}

Symmetrizing this relation with respect to $a_1, a_2$, we get:

\begin{multline}\tag{5.3}
  2(a_1 \circ a_2) \vee a_3 \otimes [b_1,b_2] \wedge b_3                      \\
+ ([a_3,a_1] \vee a_2 - [a_2,a_3] \vee a_1) \otimes 
((b_3 \circ b_1) \wedge b_2 - (b_2 \circ b_3) \wedge b_1)            \\
+ ((a_3 \circ a_1) \vee a_2 + (a_2 \circ a_3) \vee a_1) \otimes 
([b_3,b_1] \wedge b_2 + [b_2,b_3] \wedge b_1) \equiv 0.
\end{multline}

Cyclic permutations of $a_1, a_2, a_3$ in the last relation yield:

\begin{multline}\notag
((a_1 \circ a_2) \vee a_3 + (a_3 \circ a_1) \vee a_2 + (a_2 \circ a_3) \vee a_1)  \\
\otimes ([b_1,b_2] \wedge b_3 + [b_3,b_1] \wedge b_2 + [b_2,b_3] \wedge b_1) \equiv 0.
\end{multline}


This relation, in its turn, evidently implies

\begin{equation}\tag{5.4}
A \wedge A \otimes ([b_1,b_2] \wedge b_3 + [b_3,b_1] \wedge b_2 + [b_2,b_3] \wedge b_1) \equiv 0.
\end{equation}

Now rewriting (5.3) modulo (5.4) and substituting $a_3 = 1$ and $b_2 = 1$, we
get, respectively:

\begin{equation}\tag{5.5}
(a_1 \vee a_2 - 1 \vee a_1 \circ a_2) \otimes [B,B] \wedge B \equiv 0
\end{equation}

\noindent and

\begin{equation}\notag
([a_3,a_1] \vee a_2 - [a_2,a_3] \vee a_1) \otimes (b_1 \wedge b_3 + 1 \wedge b_1 \circ b_3) \equiv 0.
\end{equation}

Symmetrizing the last relation with respect to $b_1, b_3$, one gets:

\begin{equation}\tag{5.6}
([a_3,a_1] \vee a_2 - [a_2,a_3] \vee a_1) \otimes B \wedge B \equiv 0.
\end{equation}

Particularly, taking in (5.6) $a_2 = 1$, one gets

\begin{equation}\tag{5.7}
1 \vee [A,A] \otimes B \wedge B \equiv 0.
\end{equation}

Now, (5.2) is equivalent modulo (5.4)--(5.6) to

\begin{align}
&[a_1, a_2] \vee a_3 \otimes (b_1b_2 \wedge b_3 + b_3b_1 \wedge b_2 + b_2b_3 \wedge b_1) \notag \\
&+ 1 \vee ((a_1 \circ a_2) \circ a_3 - (a_2 \circ a_3) \circ a_1) \otimes 
          [b_1,b_2] \wedge b_3                                             \tag{5.8} \\
&+ 1 \vee ((a_3 \circ a_1) \circ a_2 - (a_2 \circ a_3) \circ a_1) \otimes 
          [b_3,b_1] \wedge b_2 \equiv 0.                                   \notag
\end{align}

Taking into account the identity

\begin{equation}\notag
(a \circ b) \circ c - (a \circ c) \circ b = \frac 14 [a,[b,c]]
\end{equation}

\noindent (cf. \cite{J}, p.37), and (5.7), the relation (5.8), in its turn, is equivalent to

\begin{equation}\tag{5.9}
[A,A] \vee A \otimes T(B) \equiv 0.
\end{equation}

Putting together (5.4)--(5.6) and (5.9), we get exactly the statement of the Lemma.
\end{proof}

\begin{lemma}\label{5.3}\hfill
\begin{enumerate}
\item $p(Ker\,d)/p(Im\,d) \simeq F(A,B)$.
\item $(Ker\,p \cap Ker\,d)/(Ker\,p \cap Im\,d) \simeq F(B,A)$.
\end{enumerate}
\end{lemma}


\begin{proof}
(1) is derived from Lemmas 5.1 and 5.2 after a number of routine transformations.

(2) Define a projection $p^\prime: \wedge^2 (A \otimes B) \to \wedge^2 A \otimes S^2 B$. 
Due to an obvious fact that $p^\prime$ is the identity on $Ker\,p = A \wedge A \otimes B \vee B$, 
we have an isomorphism

\begin{equation}\notag
\frac{Ker\,d \cap Ker\,p}{Ker\,p \cap Im\,d} \simeq 
\frac{p^\prime(Ker\,d \cap Ker\,p)}{p^\prime(Ker\,p \cap Im\,d)} =
\frac{p^\prime(Ker\,d)}{p^\prime(Im\,d)}
\end{equation}

But the right-hand term here is computed as in the part (1), up to permutation of 
$A$ and $B$.
\end{proof}

Now Theorem 5.1 follows immediately from (5.1) and Lemma 5.3.

\bigskip
\remark 
Taking in Theorem 5.1 $B = M_n(K)$, we get, after a series of 
elementary transformations, an isomorphism

\begin{equation}\notag
H_2(gl_n(A)) \simeq HC_1(A) \oplus \wedge^2 (A/[A, A]).
\end{equation}

Using the Hochschild-Serre spectral sequence associated with central extension

$$\notag
0 \to sl_n(A) \to gl_n(A) \to A/[A,A] \to 0
$$

\noindent of Lie algebras, we derive

$$\notag
H_2(sl_n(A)) \simeq HC_1(A)
$$

\noindent which is a result of C. Kassel and J.-L. Loday \cite{KL}.

\bigskip\bigskip
\begin{flushright}\parbox{2.6in}{
Paul Zusmanovich

Department of Mathematics

Bar-Ilan University

Ramat-Gan 52900, Israel

e-mail: \texttt{zusman@bimacs.cs.biu.ac.il}
}\end{flushright}

\end{document}